\documentclass{birkjour}
\usepackage{amsfonts,amsmath,amsthm,amscd,amssymb,latexsym,amsbsy,pb-diagram,caption,url}
\usepackage[space, noadjust]{cite}
\usepackage[mathscr]{eucal}
\usepackage[T2A]{fontenc}
\usepackage[cp1251]{inputenc}
\usepackage[english]{babel}
\usepackage{tikz} 

\newtheorem{thm}{Theorem}[section]
\newtheorem{cor}[thm]{Corollary}
\newtheorem{lem}[thm]{Lemma}
\newtheorem{prop}[thm]{Proposition}

\theoremstyle{definition}
\newtheorem{defn}[thm]{Definition}
\theoremstyle{remark}
\newtheorem{rem}[thm]{Remark}
\numberwithin{equation}{section}
\DeclareMathOperator{\diam}{diam}

\DeclareMathOperator{\Iso}{Iso}
\DeclareMathOperator{\id}{id}
\DeclareMathOperator{\Sym}{Sym}
\DeclareMathOperator{\Fix}{Fix}

\newcommand{\Sp}[1]{\operatorname{Sp}(#1)}

\begin{document}

\title[How rigid the finite ultrametric spaces can be?]{How rigid the finite ultrametric spaces can be?}

\author{O. Dovgoshey}
\address{
Division of Applied Problems in Contemporary Analysis \\
Institute of Mathematics of NASU \\ Tereshenkivska str. 3\\
Kyiv 01601\\
Ukraine
}

\email{aleksdov@mail.ru}

\author{E. Petrov}

\address{Division of Applied Problems in Contemporary Analysis\\
Institute of Mathematics of NASU\\ Tereshenkivska str. 3\\
Kyiv 01601\\
Ukraine}

\email{eugeniy.petrov@gmail.com}

\author{H.-M. Teichert}

\address{Institute of Mathematics\\
University of L\"ubeck\\
Ratzeburger Allee 160\\
23562 L\"ubeck\\
Germany}

\email{teichert@math.uni-luebeck.de}

\thanks{The research of the first and second authors was supported as a part of EUMLS project with grant agreement PIRSES -- GA -- 2011 -- 295164.}

\subjclass{Primary 54E35}

\keywords{finite ultrametric space, rigidness of metric spaces, extremal problems in finite ultrametric spaces, fixed point, representing tree of ultrametric space.}

\begin{abstract}
A metric space $X$ is rigid if the isometry group of $X$ is trivial. The finite ultrametric spaces $X$ with $|X| \geq 2$ are not rigid since for every such $X$ there is a self-isometry having exactly $|X|-2$ fixed points. Using the representing trees we characterize the finite ultrametric spaces $X$ for which every self-isometry has at least $|X|-2$ fixed points. Some other extremal properties of such spaces and related  graph theoretical characterizations are also obtained.
\end{abstract}

\maketitle
\section{Introduction}
Recall some definitions from the theory of metric spaces. A \textit{metric} on a set $X$ is a function $d\colon X\times X\rightarrow \mathbb{R}^+$, $\mathbb R^+ = [0,\infty)$, such that for all $x$, $y$, $z \in X$:
\begin{itemize}
\item [(i)] $d(x,y)=d(y,x)$,
\item [(ii)] $(d(x,y)=0)\Leftrightarrow (x=y)$,
\item [(iii)] $d(x,y)\leq d(x,z) + d(z,y)$.
\end{itemize}
A metric space $(X, d)$ is \emph{ultrametric} if the \emph{strong triangle inequality}
$$
d(x,y)\leq \max \{d(x,z), d(z,y)\}
$$
holds for all $x$, $y$, $z \in X$. In this case $d$ is called an ultrametric on $X$ and $(X, d)$ is an ultrametric space. The \emph{spectrum} of a metric space $(X,d)$ is the set
$$
\operatorname{Sp}(X)=\{d(x,y)\colon x,y \in  X\}.
$$

In 2001 at the Workshop on General Algebra~\cite{WGA} the attention of experts on the theory of lattices was guided to the following problem of I.~M.~Gelfand: using graph theory describe up to isometry all finite ultrametric spaces. An appropriate representation of ultrametric spaces by rooted trees was proposed in~\cite{GurVyal(2012)}, \cite{GV}, \cite{H04}, \cite{Lem}.

An application of the representation from~\cite{GurVyal(2012)}, \cite{GV} is a structural characteristic of all finite ultrametric spaces $X$ for which the Gomory-Hu inequality
$$
|\operatorname{Sp}(X)| \leq |X|
$$
becomes an equality (see~\cite{DPT(P-adic)} and \cite{PD(UMB)}). The purpose of this papers is to describe the structure of finite ultrametric spaces which have maximum rigidity.

Recall that a \textit{graph} is a pair $(V,E)$ consisting of a nonempty set $V$ and a (probably empty) set $E$  elements of which are unordered pairs of different points from $V$. For a graph $G=(V,E)$, the sets $V=V(G)$ and $E=E(G)$ are called \textit{the set of vertices} and \textit{the set of edges}, respectively. A graph is \emph{complete} if $\{x,y\} \in E(G)$ for all distinct $x, y \in V(G)$. A \emph{path} in a graph $G$ is a subgraph $P$ of $G$ for which
$$
V(P)=\{x_0,x_1,...,x_k\}, \quad E(P) =\{\{x_0,x_1\},...,\{x_{k-1},x_k\}\},
$$
where all $x_i$ are distinct. A graph $G$ is \emph{connected} if any two distinct vertices of $G$ can be joined by a path. A finite graph $C$ is a cycle if $|V(C)| \geq 3$ and there exists an enumeration $(v_1, \ldots, v_n)$ of its vertices such that
$$
(\{v_i, v_j\} \in E(C)) \Leftrightarrow (|i-j|=1 \text{ or } |i-j|=n-1).
$$
A connected graph without cycles is called a \emph{tree}. A tree $T$ may have a distinguished vertex called the \emph{root}; in this case $T$ is called a \emph{rooted tree}. For a rooted tree $T$, we denote by $L_T$ the set of leaves of $T$. We denote by $(G, v, l)$ a graph $G$ with a \emph{distinguished vertex} $v \in V(G)$ and a \emph{labeling function} $l\colon V(G) \to L$. In what follows we usually suppose that the set $L$ coincides with $\mathbb R^{+}$.

Let $k\geq 2$. A graph $G$ is called \emph{complete $k$-partite} if its vertices can be divided into $k$ disjoint nonempty sets $X_1,...,X_k$ so that there are no edges joining the vertices of the same set $X_i$ and any two vertices from different $X_i,X_j$, $1\leq i,j \leq k$ are adjacent. In this case we write $G=G[X_1,...,X_k]$. We shall say that $G$ is a {\it complete multipartite graph} if there exists $k \geq 2$ such that $G$ is complete $k$-partite, cf. \cite{Di}.

\section{The representing trees of finite ultrametric spaces}

For every metric space $(X, d)$ we write
$$
\diam{X} = \sup\{d(x,y): x, y \in X\}.
$$

\begin{defn}[\cite{DDP(P-adic)}]\label{d2}
Let $(X,d)$ be a finite ultrametric space, $|X|\geqslant 2$. Define the \emph{diametrical} graph $G_X$ as follows $V(G_X)=X$ and, for all $u$, $v \in X$,
$$
(\{u,v\}\in E(G_X))\Leftrightarrow(d(u,v)=\diam X).
$$
\end{defn}

\begin{lem}[\cite{DDP(P-adic)}]\label{t13}
Let $(X,d)$ be a finite ultrametric space, $|X|\geq 2$. Then $G_X=G_X[X_1,...,X_k]$ for some $k\geq 2$.
\end{lem}

With a finite nonempty ultrametric space $(X, d)$, we can associate a labeled rooted tree $T_X$ by the following rule. The root of $T_X$ is, by definition, the set $X$.  If $X=\{x\}$ is a one-point set, then $T_X$ is a tree consisting of one node $\{x\}$ which has the label $0$. Let $|X|\geq 2$. According to Lemma~\ref{t13} we have $G_X = G_X[X_1,...,X_k]$. In this case, we set that, the root of the tree $T_X$ is labeled by $\diam X$ and $T_X$ has $k$ nodes $X_1,...,X_k$ of the first level with the labels
\begin{equation}\label{e2.1}
l(X_i):=\diam X_i
\end{equation}
$i = 1,...,k$. The nodes of the first level indicated by zero are leaves, and those indicated by positive numbers are internal nodes of $T_X$. If the first level has no internal nodes, then the tree $T_X$ is constructed. Otherwise, by repeating the above-described procedure with $X_i$ corresponding to internal nodes of the first level instead of $X$, we obtain the nodes of the second level, etc. Since $X$ is finite, all vertices on some level will be leaves, and the construction of $T_X$ is completed.

The above-constructed labeled rooted tree $T_X$ is called the \emph{representing tree} of the ultrametric space $(X, d)$.

Let $(T, v^*)$ be a rooted tree. For every node $u^*$ of $(T, v^*)$ define a rooted subtree $T_{u^*}$ of $(T, v^*)$ as follows: $u^*$ is the root of $T_{u^*}$ and a vertex $w \in T$ belongs to $V(T_{u^*})$ if and only if $u^*$ lies on the path joining $v^*$ and $w$ in $T$, moreover
$$
(\{u, v\} \in E(T_{u^*})) \Leftrightarrow (\{u, v\} \in E(T)).
$$
for all $u$, $v \in V(T_{u^*})$.

The following lemma was formulated in~\cite{PD(UMB)} for finite ultrametric spaces $X$ satisfying the equality $|X| = |\Sp{X}|$ but its proof is also true for arbitrary finite ultrametric spaces.

\begin{lem}\label{l2.3}
Let $(X, d)$ be a finite ultrametric space, $|X|\geq 2$, and let $a$ and $b$ be two different leaves of the tree $T_X$. If $(x_1, x_2, . . . , x_n)$, $x_1=a$, $x_n=b$, is the path joining $a$ and $b$ in $T_X$, then
\begin{equation}\label{e2.2}
d(a, b) = \max\limits_{1\leq i \leq n} l({x}_i).
\end{equation}
\end{lem}

Let $(X, d)$ be a metric space. Recall that a subset $B$ of $X$ is a \emph{ball} in $(X,d)$ if there is $r \geq 0$ and $t \in X$ such that
$$
B=\{x\in X\colon d(x,t)\leq r\}.
$$
In this case we write $B=B_r(t)$. By $\mathbf{B}_X$ we denote the set of all balls in $(X, d)$.

The proof of the next lemma can be found in \cite{P(TIAMM)} but we reproduce it here for the convenience of the reader.

\begin{lem}\label{l2.4}
 Let $(X,d)$ be a finite ultrametric space  with representing tree  $T_X$, $|X|\geq 2$. Then
\begin{itemize}
\item [(i)] $L_{T_v}\in \mathbf{B}_X$ holds for every node $v\in V(T_X)$.
\item [(ii)] For every $B \in \mathbf{B}_X$ there exists a node $v$ such that $L_{T_v}=B$.
\end{itemize}
\end{lem}
\begin{proof}
\textbf{(i)} Let $v\in V(T_X)$ and $t\in L_{T_v}$. Consider the ball
$$
B_{l(v)}(t)=\{x \in X\colon d(x,t)\leq l(v)\}.
$$
Let $t_1 \in L_{T_v}$ such that $t_1 \neq t$. Since $T_v$ contains a path joining $t$ and $t_1$, according to Lemma~\ref{l2.3} we have $d(t,t_1)\leq l(v)$. The inclusion $L_{T_v}\subseteq {B}_{l(v)}(t)$ is proved. Conversely, suppose there exists $t_0\in B_{l(v)}(t)$ such that $t_0 \notin L_{T_v}$. Let us consider the path $(t_0,v_1,...,v_n,t)$. From $t_0\notin L_{T_v}$ it follows that $\max\limits_{1\leq i \leq n}l(v_i)>l(v)$, i.e. $d(t_0,t)>l(v)$, we have a contradiction.

\textbf{(ii)} Let $t\in X$ and $B=B_r(t)$, where $r=\diam B$. Let $x$, $y \in B$ with $d(x,y) = r$. Let us consider the path $(v_1,...,v_n)$ with $v_1=x$ and $v_n=y$ in the tree $T_X$. According to Lemma~\ref{l2.3} we have $d(x,y) = \max\limits_{1\leq i \leq n}l(v_i)$. Let $i$ be an index such that $\max$ here is attained. The proof of the equality $L_{T_{v_i}}=B$ is analogous to the proof of (i).
\end{proof}

Let $(G^{i}, v^{i}, l^{i})$, $i=1, 2$, be a labeled graphs with the distinguished vertices $v^1$, $v^2$ and a common set $L$ of labels. A bijective function $f: V(G^1) \to V(G^2)$ is an isomorphism of $G^{1}$ and $G^{2}$ if
$$
(\{x,y\} \in E(G^1)) \Leftrightarrow (\{f(x),f(y)\} \in E(G^2))
$$
for all $x$, $y \in V(G^1)$. If, in addition, we have $f(v^1)=v^2$, then $f$ is an isomorphism of $(G^1, v^1)$ and $(G^2, v^2)$. The isomorphism $f$ of $(G^1, v^1)$ and $(G^2, v^2)$ is called an isomorphism of the $(G^1, v^1, l^1)$ and $(G^2, v^2, l^2)$ if $l^1(v) = l^2(f(v))$ for every $v \in V(G^1)$.

\begin{defn}\label{d2.5}
Let $(X, d)$ and $(Y, \rho)$ be metric spaces. A bijective function $f: X \to Y$ is an \emph{isometry} if
$$
d(x, y) = \rho(f(x), f(y))
$$
holds for all $x$, $y \in X$.
\end{defn}

Two metric spaces $(X, d)$ and $(Y, \rho)$ are \emph{isometric} if there is an isometry $f: X \to Y$.

\begin{thm}\label{t2.6}
Let $(X, d)$ and $(Y, \rho)$ be finite nonempty ultrametric spaces. Then $(X, d)$ and $(Y, \rho)$ are isometric if and only if the labeled rooted trees $T_X$ and $T_Y$ are isomorphic.
\end{thm}
\begin{proof}
An isomorphism of $T_X$ and $T_Y$ for isometric $X$ and $Y$ can be inductively constructed if we use of the definition of the representing trees given above. The converse statement follows from Lemma~\ref{l2.3}.
\end{proof}

For the relationships between ultrametric spaces and the leaves or the ends of certain trees see also~\cite{DLW, Fie, GurVyal(2012), GV, Hol, H04, WGA, Lem}.

Let $(X, d)$ be a finite metric space and let $B_1$, $B_2$ and $B$ be some balls in~$(X, d)$. We shall say that $B$  \emph{lies between} $B_1$ and $B_2$ if
$$
B_1 \subseteq B \subseteq B_2 \text{ or } B_2 \subseteq B \subseteq B_1.
$$
If $B$ lies between $B_1$ and $B_2$, we write $B \in [B_1, B_2]$.

\begin{defn}\label{d2.7}
Let us define a graph $\Gamma_X$ by the rule: $V(\Gamma_X) = \mathbf{B}_X$ and, for all $B_1$, $B_2 \in \mathbf{B}_X$, $\{B_1, B_2\} \in E(\Gamma_X)$ if and only if
\begin{enumerate}
\item $B_1 \subseteq B_2$ or $B_2 \subseteq B_1$, and
\item $(B \in [B_1, B_2]) \Rightarrow (B=B_1 \text{ or } B = B_2)$ for every $B \in \mathbf{B}_X$.
\end{enumerate}
\end{defn}

For any labeled rooted tree $T_X = (T_X, X, l)$ we write $\overline{T}_X = (T_X, X)$. In the following theorem we presuppose that $X \in \mathbf{B}_X$ is a distinguished point in $\Gamma_X$.
\begin{thm}\label{t2.7}
Let $(X, d)$ be a finite nonempty metric space. Then the graph $\Gamma_X$ is a tree if and only if $(X, d)$ is ultrametric. If $(X, d)$ is ultrametric, then $(\Gamma_X, X)$ is isomorphic to the rooted tree $\overline{T}_X$ with the isomorphism $V(\overline{T}_X)\ni u\mapsto L_{T_u} \in V(\Gamma_X)$.
\end{thm}
\begin{proof}
Suppose that $(X, d)$ is not an ultrametric space. Then there are distinct $x_1$, $x_2$, $x_3 \in X$ such that
\begin{equation}\label{e2.3}
d(x_1, x_2) > \max\{d(x_1, x_3), d(x_3, x_2)\}.
\end{equation}
Let $B^1 = B_{r_1}(x_1)$ with $r_1 = d(x_1, x_3)$ and let $B^2 = B_{r_2}(x_2)$  with $r_2 = d(x_2, x_3)$. It is evident that
\begin{equation}\label{e2.4}
\{x_3\} \subseteq B^1 \cap B^2 \subseteq B^1 \cup B^2 \subseteq X.
\end{equation}
Using \eqref{e2.4} and the finiteness of $\mathbf{B}_X$ we can find $B^0 \in \mathbf{B}_X$ such that $B^0 \subseteq B^1 \cap B^2$ and
$$
(B^0 \subseteq B \subseteq B^1 \cap B^2) \Rightarrow (B^0 = B)
$$
for every $B\in \mathbf{B}_X$. Similarly there exists $B^3 \in \mathbf{B}_X$ such that
$$
B^3 \supseteq B^1 \cup B^2
$$
and
$$
(B^1 \cup B^2 \subseteq B \subseteq B^3) \Rightarrow (B = B^3)
$$
for every $B\in \mathbf{B}_X$.

For all distinct balls $C$, $D \in \mathbf{B}_X$ satisfying $C \subseteq D$, there is a chain
$$
L(C, D) = \{A_1, \ldots, A_n\} \subseteq \mathbf{B}_X
$$
such that
$$
A_1 = C, \quad A_n = D, \quad A_1 \subseteq A_2 \subseteq \ldots \subseteq A_n
$$
and, for every $B \in \mathbf{B}_X$ and $i=1$, $\ldots$, $n-1$,
\begin{equation}\label{e2.5}
(A_i \subseteq B \subseteq A_{i+1}) \Rightarrow (B = A_i \text{ or } B = A_{i+1}).
\end{equation}
Using~\eqref{e2.3} we see that $x_1 \notin B^2$ and $x_2 \notin B^1$. Consequently the balls $B^0$, $B^1$, $B^2$, $B^3$ are pairwise distinct, so that there exist the chains $L(B^0, B^1)$, $L(B^0, B^2)$, $L(B^1, B^3)$ and $L(B^2, B^3)$. Let $L(B^0, B^1) = (A_1, \ldots, A_n)$ and $L(B^1, B^3) =  (A_n, \ldots, A_{n+m})$. We claim that $L(B^0, B^1, B^3) := (A_1, \ldots, A_n, \ldots, A_{n+m})$ is a path in the graph $\Gamma_X$ joining $B^0=A_1$ and $B^3 = A_{n+m}$.

Indeed, from the definition of $L(B^0, B^1)$ and $L(B^1, B^3)$ it follows that $A_i \neq A_j$ for any distinct $i$, $j \in \{1, \ldots, n+m\}$. Moreover, \eqref{e2.5} implies that $\{A_i, A_{i+1}\} \in E(\Gamma_X)$ for every $i \in \{1, \ldots, n+m-1\}$. Hence $L(B^0, B^1, B^3)$ is a path in $\Gamma_X$. Inequality~\eqref{e2.3} implies that $x_1 \notin B^2$ and $x_2 \notin B^1$, consequently $B^1 \nsubseteqq B^2$ and $B^2 \nsubseteqq B^1$. Since for every vertex $B\in V(L(B^0, B^1, B^3))$ we have
$$
B \subseteq B^1 \text{ or } B^1 \subseteq B,
$$
the ball $B^2$ is not a vertex of the path $L(B^0, B^1, B^3)$. Similarly we can construct a path $L(B^0, B^2, B^3)$ such that $B^1 \notin V(L(B^0, B^2, B^3))$. Consequently the vertices $B^0$ and $B^3$ can be joined by two different paths in $\Gamma_X$. Hence $\Gamma_X$ is not a tree.

Suppose now that $(X, d)$ is an ultrametric space. We must show that $\Gamma_X$ is a tree and that $(\Gamma_X, X)$ is isomorphic to the rooted tree $\overline{T}_X$. By Lemma~\ref{l2.4}, for every $v \in V(\Gamma_X)$, we have $L_{T_v} \in \mathbf{B}_X$. Consequently the mapping
\begin{equation}\label{e2.6}
V(\overline{T}_X) \ni v \mapsto L_{T_v} \in V(\Gamma_X)
\end{equation}
is correctly defined. Hence it is sufficient to show that this mapping is an isomorphism of the rooted tree $\overline{T}_X$ and $(\Gamma_X, X)$.

Using Lemma~\ref{l2.4} again we obtain that~\eqref{e2.6} is bijective. It is also clear that the root of $\overline{T}_X$ corresponds to $X$ under mapping~\eqref{e2.6}. It still remains to show that
\begin{equation}\label{e2.7}
(\{v_1, v_2\} \in E(\overline{T}_X)) \Leftrightarrow (\{L_{T_{v_1}}, L_{T_{v_2}}\} \in E(\Gamma_X))
\end{equation}
holds for all $v_1$, $v_2 \in V(\overline{T}_X)$.

Write $B_1 = L_{T_{v_1}}$ and $B_2 = L_{T_{v_2}}$. Then it follows from the definition of the rooted subtrees that $B_1 \subseteq B_2$ if and only if $v_1$ is a node of the tree $T_{v_2}$. Moreover if $B_1 \subseteq B_2$, then
$$
(B_1 \subseteq B \subseteq B_2) \Rightarrow (B_2 = B \text{ or } B_1 = B)
$$
holds for all $B \in \mathbf{B}_X$ if and only if $v_1$ is a direct successor of $v_2$. Statement~\eqref{e2.7} follows.
\end{proof}

\begin{figure}[h]
\begin{tikzpicture}[scale=0.8]
\draw (0,0) -- (1,2) -- (2,4) -- (3,2) -- (4,0);
\draw (1,2) -- (2,0); \draw (3,2) -- (2,0); \fill[black] (2,0) circle(2pt);
\fill[black] (0,0) circle(2pt) -- (1,2) circle(2pt) -- (2,4) circle(2pt) -- (3,2) circle(2pt) -- (4,0) circle(2pt);
\end{tikzpicture}
\caption{For all nonultrametric metric triangles $X$ their graphs $\Gamma_X$ are isomorphic to the depicted graph.}
\label{fig1*}
\end{figure}
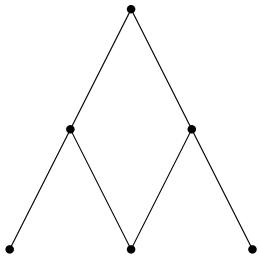

Since a connected graph $G$ is a tree if and only if
$$
|V(G)|=|E(G)|+1
$$
(see, for example, \cite[Corollary~1.5.3]{Di}), Theorem~\ref{t2.7} implies the following

\begin{cor}\label{c2.9}
Let $X$ be a finite nonempty metric space. Then $X$ is an ultrametric space if and only if
$$
|V(\Gamma_X)| = |E(\Gamma_X)| + 1.
$$
\end{cor}

Recall that a graph $H$ is a \emph{spanning subgraph} of a graph $G$ if $V(G)=V(H)$ and $E(H) \subseteq E(G)$. A graph is connected if and only if it has a spanning tree \cite[Theorem~4.6]{BM}.

\begin{cor}\label{c2.10}
Let $X$ be a finite nonempty metric space and $Y$ be a finite nonempty ultrametric space. If $|\mathbf{B}_X| = |\mathbf{B}_Y|$, then the inequality
\begin{equation}\label{e2.8}
|E(\Gamma_X)| \geq |E(\Gamma_Y)|
\end{equation}
holds. Furthermore, equality in~\eqref{e2.8} occurs if and only if $X$ is ultrametric.
\end{cor}
\begin{proof}
It is easy to see that $\Gamma_X$ is connected. Let $T$ be a spanning tree of $\Gamma_X$. Then we have
\begin{equation}\label{e2.9}
|E(\Gamma_X)| \geq |E(T)|.
\end{equation}
Since
$$
|V(T)| = |V(\Gamma_Y)| = |V(\Gamma_X)| = |\mathbf{B}_X| = |\mathbf{B}_Y|
$$
and $|V(T)| = |E(T)|+1$ and $|V(\Gamma_Y)|=|E(\Gamma_Y)| + 1$, inequality \eqref{e2.9} implies \eqref{e2.8}. Now using Corollary~\ref{c2.9} we obtain that
$$
|E(\Gamma_X)| = |E(\Gamma_Y)|
$$
holds if and only if $X$ is an ultrametric space.
\end{proof}

\section{Finite ultrametric spaces having maximum rigidity}

Let $(X, d)$ be a metric space and let $\operatorname{Iso}(X)$ be the group of isometries of $(X, d)$. We say that $(X, d)$ is \emph{rigid}  if $|\operatorname{Iso}(X)|=1$. It is clear that $(X, d)$ is rigid if and only if $g(x)=x$ for every $x \in X$ and every $g \in \Iso(X)$.

For every self-map $f\colon X\to X$ we denote by $\Fix(f)$ the set of fixed points of $f$. Using this denotation we obtain that a finite metric space $(X, d)$ is rigid if and only if
$$
\min_{g \in \operatorname{Iso}(X)} |\Fix(g)|=|X|.
$$

\begin{prop}\label{pr3.1}
Let $(X, d)$ be a finite ultrametric space with $|X| \geq 2$. Then $(X, d)$ is nonrigid.
\end{prop}
\begin{proof}
It is sufficient to construct $g \in \Iso(X)$ such that
\begin{equation}\label{e3.1*}
|\Fix(g)| \leq |X|-2.
\end{equation}
Since $X$ is finite, the representing tree $T_X$ contains an internal node $v$ such that the direct successors of $v$ are leaves of $T_v$. We have the inequality
$$
|L_{T_v}| \geq 2
$$
because $v$ is internal. Hence there is a fixed point free bijection $\psi: L_{T_v} \to L_{T_v}$, i.e.
\begin{equation}\label{e3.2*}
|\Fix(\psi)|=0
\end{equation}
holds. The leaves of $T_X$ is the one-point subsets of $X$. Identifying the leaves with their respective points of $X$ we can define a bijection $g: X\to X$ as
\begin{equation}\label{e3.3*}
g(x) =
\left\{\begin{array}{ll}
\psi(x) & \text{ if } x \in L_{T_v}\\
x & \text{ if } x \in X\setminus L_{T_v}.
\end{array}
\right.
\end{equation}
Lemma~\ref{l2.3} implies that $g \in \Iso(X)$. From \eqref{e3.2*} and \eqref{e3.3*} we obtain the equality
$$
|\Fix(g)| = |X\setminus L_{T_v}| = |X| - |L_{T_v}|.
$$
Since $|L_{T_v}| \geq 2$, inequality~\eqref{e3.1*} follows.
\end{proof}

\begin{rem}
The Fibonacci space is an interesting example of a compact infinite rigid ultrametric space. (See \cite{H04} and \cite{BH2} for some interesting properties of the Fibonacci space.)
\end{rem}

If a metric space $(X, d)$ is finite, nonempty and nonrigid, then the inequality
\begin{equation}\label{e3.1}
\min_{g \in \operatorname{Iso}(X)} |\Fix(g)| \leq |X|-2
\end{equation}
holds, because the existence of $|X|-1$ fixed points for $g \in \operatorname{Iso}(X)$ implies that $g$ is identical.

The quantity $\min_{g \in \operatorname{Iso}(X)} |\Fix(g)|$ can be considered as a measure for ``rigidness'' for finite metric spaces $(X, d)$. Thus the finite ultrametric spaces satisfying the equality
\begin{equation}\label{e3.2}
\min_{g \in \operatorname{Iso}(X)} |\Fix(g)| = |X|-2,
\end{equation}
are as rigid as possible. Let us denote by $\mathfrak{R}$ the class of all finite ultrametric spaces $(X, d)$ which satisfy this equality.

\begin{lem}\label{l3.2}
Let $(X, d)$ be a nonempty finite ultrametric space. Then for every set of partial self-isometries
$$
S = \{\psi_i: B_i \to B_i: i \in I\},
$$
where $I$ is an index set, such that
$$
B_i \in \mathbf{B}_X, \quad B_i \cap B_j = \varnothing
$$
for all distinct $i$, $j \in I$, there exists an isometry $\psi \in \Iso(X)$ for which the restriction $\psi|_{B_i}$ equals $\psi_i$ for every $i \in I$.
\end{lem}
\begin{proof}
Let us define a required $\psi$ by the rule
$$
\psi(x) = \left\{
\begin{array}{ll}
\psi_i(x), & \text{if } x \in B_i, i \in I,\\
x , & \text{if } x \in X \setminus \left(\bigcup\limits_{i \in I}B_i\right).
\end{array}
\right.
$$
It follows from Lemma~\ref{t13} that $\psi \in \Iso(X)$.
\end{proof}

Recall that a rooted tree is \emph{strictly $n$-ary} if every internal node has exactly $n$ children. In the case $n=2$ such tree is called \emph{strictly binary}. A level of a vertex of a rooted tree $(T, v^*)$ can be defined by the following inductive rule: The level of the root $v^*$ is zero and if $u \in V(T)$ has a level $x$, then every direct successor of $u$ has the level $x+1$.

\begin{thm}\label{th3.3}
Let $(X, d)$ be a finite ultrametric space with $|X| \geq 2$. Then the following statements are equivalent:
\begin{itemize}
\item [(i)] $(X, d) \in \mathfrak{R}$;
\item [(ii)] $|\Iso(X)|=2$;
\item [(iii)] $T_X$ is strictly binary with exactly one inner node at each level except the last level.
\end{itemize}
\end{thm}
\begin{proof}
\textbf{(i)}$\Rightarrow$\textbf{(ii)}. Let (i) hold and
$$
|\Iso(X)| \neq 2.
$$
By Proposition~\ref{pr3.1}, the ultrametric space $(X, d)$ is nonrigid. Consequently there exist $\psi_1$, $\psi_2 \in \Iso(X)$ such that
$$
\psi_1 \neq \psi_2 \quad\text{and}\quad \psi_1 \neq \id_X \neq \psi_2,
$$
where $\id_X$ is the identical mapping of $X$. Since $(X, d) \in \mathfrak{R}$, we can find the sets $\{x_1^1, x_2^1\}$ and $\{x_1^2, x_2^2\}$ such that $\psi_i(x_1^i) = x_2^i$, $\psi_i(x_2^i) = x_1^i$, $i = 1$, $2$ and
\begin{equation}\label{e3.7}
\{x_1^1, x_2^1\} \neq \{x_1^2, x_2^2\}.
\end{equation}
Note that if $\{x_1^1, x_2^1\} = \{x_1^2, x_2^2\}$, then $\psi_1 = \psi_2$ because all points of the set $X \setminus (\{x_1^1, x_2^1\} \cup \{x_1^2, x_2^2\})$ are fixed points of $\psi_1$ and $\psi_2$.  A short calculation shows that the set $\{x_1^1, x_2^1\} \cup \{x_1^2, x_2^2\}$ contains no fixed points of the composition $\psi_1\circ \psi_2$ and that
$$
\psi_1\circ \psi_2(x) = x
$$
for every $x \in X \setminus (\{x_1^1, x_2^1\} \cup \{x_1^2, x_2^2\})$. Hence we have
$$
|\Fix(\psi_1\circ \psi_2)| = |X|-3 \text{ or } |\Fix(\psi_1\circ \psi_2)| = |X|-4,
$$
that contradicts equality~\eqref{e3.2}.

\textbf{(ii)}$\Rightarrow$\textbf{(iii)}. Let $|\Iso(X)| = 2$. We must prove (iii). First, we prove that:
\begin{itemize}
\item[$(s_1)$] For every inner node $v$ of $T_X$, the set $Ch(v)$ of children of $v$ contains at most one inner node and at most two leaves.
\end{itemize}

Let $v$ be an inner node of $T_X$ and let $v_1$, $v_2 \in Ch(v)$ be distinct inner nodes. Then the balls $B_1 = L_{T_{v_1}}$ and $B_2 = L_{T_{v_2}}$ are disjoint and the inequalities
$$
|B_1| \geq 2, \quad |B_2| \geq 2
$$
hold. By Proposition~\ref{pr3.1} the metric spaces $(B_1, d)$ and $(B_2, d)$ are nonrigid. Hence there are $\psi_1 \in \Iso(B_1)$ and $\psi_2 \in \Iso(B_2)$ such that
$$
\psi_1 \neq \id_{B_1}, \text{ and } \psi_2 \neq \id_{B_2},
$$
where $\id_{B_i}$ is the identical mapping of the set $B_i$, $i=1$, $2$. By Lemma~\ref{l3.2} there are $\psi^1$, $\psi^2 \in \Iso(X)$ such that
$$
\psi^1|_{B_1}=\psi_1, \quad \psi^1|_{B_2}=\id_{B_2}, \quad \psi^2|_{B_1}=\id_{B_1}, \quad \psi^2|_{B_2}=\psi_2.
$$
Since $\psi^1$, $\psi^2 \in \Iso(X)$ and $\psi^1 \neq \psi^2$ and $\psi^1 \neq \id_X \neq \psi^2$, we have
$$
|\Iso(X)| \geq 3,
$$
contrary to $|\Iso(X)|=2$. The first part of $(s_1)$ is proved. In what follows we write $B = L_{T_v}$ for short. Let $G_B = G_B[X_1, \ldots, X_k]$ be the diametrical graph of the ball $B$. A leaf $\{x\}$ of $T_v$ is a child of $v$ if and only if there is $i \in \{1, \ldots, k\}$ such that $X_i = \{x\}$. Suppose that there are some three distinct leaves among the children of $v$. For certainty, we can assume that
$$
X_1 = \{x_1\}, \quad X_2 = \{x_2\} \quad\text{and}\quad X_3 = \{x_3\}.
$$
Let $S = \{x_1, x_2, x_3\}$. Using the definition of the diametrical graphs we can easily prove that every bijection $\alpha\colon S \to S$ can be extended to an isometry of $B$. Consequently, by Lemma~\ref{l3.2}, there is an extension of $\alpha$ to an isometry of $S$. Since $\Sym(S)$ (the group of symmetries of $S$) has the order $6$, we have $|\Iso(X)| \geq 6$, contrary to $|\Iso(X)| =2$. Statement $(s_1)$ follows.

To finished the proof of (iii), it suffices to show that $T_X$ does not contain any node with exactly $3$ children. Suppose contrary that $v$ is an node of $T_X$ with $3$ children. Write $B = L_{T_v}$. Then we have $G_B = G_B[X_1, X_2, X_3]$. Using $(s_1)$ we can suppose that
$$
|X_1| = |X_2| =1 \quad\text{and}\quad |X_3| \geq 2.
$$
Let $\{x_1\}=X_1$ and $\{x_2\}=X_2$, $x_1$, $x_2 \in X$, and let $S = \{x_1, x_2\}$. Let us consider $\alpha \in \Sym(S)$ and $\beta \in \Iso(X_3)$. Define $\psi: B \to B$ as
$$
\psi(x) = \left\{
\begin{array}{ll}
\alpha(x) & \text{if } x \in S,\\
\beta(x) & \text{if } x \in X_3.
\end{array}
\right.
$$
It follows directly from Definition~\ref{d2} that $\psi \in \Iso(B)$. Hence, $|\Iso(B)| \geq 4$. This inequality and Lemma~\ref{l3.2} imply the inequality $|\Iso(X)| \geq 4$, contrary to $|\Iso(X)| = 2$.

\textbf{(iii)}$\Rightarrow$\textbf{(i)}. Let (iii) hold. By Theorem~\ref{t2.6} to prove (i) it suffices to show that every self-isomorphism $f:V(T_X) \to V(T_X)$ of representing tree $T_X$ has at most two points which are not fixed points. Let us consider a self-isomorphism $f\colon V(T_X) \to V(T_X)$. The root $X$ is evidently a fixed point of $f$. If $v_1'$ and $v_2'$ are the nodes of the second level and $v_2'$ is inner, then according to (iii) $v_2'$ is a leaf and $f(v_1')=v_1'$ and $f(v_2')=v_2'$ because $f$ preserves the levels, the inner nodes and the labels. Similarly we can prove that all vertices on an arbitrary level are fixed if it is not the last level of $T_X$. Now it is sufficient to note that the last level contains exactly two vertices, because $T_X$ is strictly binary and the previous level contains exactly one inner node of $T_X$.
\end{proof}

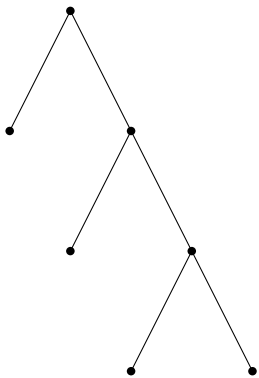
\begin{figure}[h]
\begin{tikzpicture}[scale=0.8]
\draw (0,4) -- (1,6) -- (4,0);
\fill[black] (1,6) circle(2pt);
\fill[black] (0,4) circle(2pt);
\fill[black] (2,4) circle(2pt);
\fill[black] (3,2) circle(2pt);
\fill[black] (4,0) circle(2pt);
\draw (2,4) -- (1,2);
\draw (3,2) -- (2,0);
\fill[black] (1,2) circle(2pt);
\fill[black] (2,0) circle(2pt);
\end{tikzpicture}
\caption{The representing tree $T_X$ for $X \in \mathfrak{R}$ with $|X|=4$.}
\label{fig1}
\end{figure}

Using Theorem~\ref{th3.3} we can obtain the following extremal property of ultrametric spaces belonging to $\mathfrak{R}$.

\begin{cor}\label{c3.4}
Let $X$ and $Y$ be finite ultrametric spaces with $|X|=|Y|$. If $Y \in \mathfrak{R}$, then the inequality
\begin{equation}\label{e3.8}
|\Sp{X}| \leq |\Sp{Y}|
\end{equation}
holds.
\end{cor}
\begin{proof}
Indeed, it was proved by E.~R.~Gomory and T.~C.~Hu in \cite{GomoryHu(1961)} that for every finite ultrametric space $X$ we have the inequality
$$
|\Sp{X}| \leq |X|.
$$
As was shown in \cite{PD(UMB)} the equality $|\Sp{X}| = |X|$ holds if and only if $T_X$ is strictly binary and the labels of different internal nodes are different. Note that if $Z$ is a finite ultrametric space with $u$, $v \in V(T_Z)$ and $u$ is a child of $v$, then Definition~\ref{d2}, Lemma~\ref{t13} and the definition of $T_Z$ imply the strict inequality
\begin{equation}\label{e3.9}
l(u) < l(v).
\end{equation}
Since, for every $Y \in \mathfrak{R}$, the representing tree $T_Y$ has exactly one inner node on each level expect the last level, inequality~\eqref{e3.9} shows that the labels of different internal nodes are different. Hence~\eqref{e3.8} holds if $Y \in \mathfrak{R}$.
\end{proof}

\begin{rem}
If $X$ is a finite ultrametric space and $|X| = |\Sp{X}|$, then $X$ is generally not an element of $\mathfrak{R}$ (see Figure~\ref{fig2}).
\end{rem}

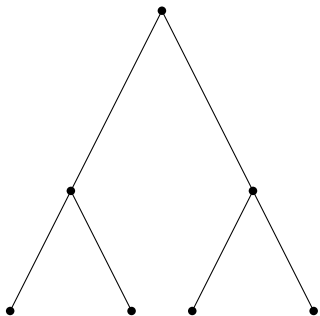
\begin{figure}[h]
\begin{tikzpicture}[scale=0.8]
\draw (0,0) -- (2.5,5) -- (5,0);
\draw (1,2) -- (2,0); \draw (4,2) -- (3,0);
\fill[black] (0,0) circle(2pt);
\fill[black] (1,2) circle(2pt);
\fill[black] (2.5,5) circle(2pt);
\fill[black] (4,2) circle(2pt);
\fill[black] (2,0) circle(2pt);
\fill[black] (3,0) circle(2pt);
\fill[black] (5,0) circle(2pt);
\end{tikzpicture}
\caption{The representing tree for an ultrametric space $X$ with $|X| = |\Sp{X}|=4$.}
\label{fig2}
\end{figure}

\section{Rigidness, stars and weak similarities}

Following~\cite{BM} denote by $K_{m,n}=K_{m,n}[X,Y]$ a complete bipartite graph such that $|X|=m$ and $|Y|=n$. In the case when $m=1$ or $n=1$ such graphs are called \emph{stars}.

The following proposition gives us the first characterization of the class $\mathfrak{R}$ by stars.

\begin{prop}\label{p4.1}
Let $(X, d)$ be a finite ultrametric space with $|X| \geq 2$. Then $X \in \mathfrak{R}$ if and only if the diametrical graphs $G_B$ are stars for all $B \in \mathbf{B}_X$ with $\diam B>0$.
\end{prop}
\begin{proof}
From Lemma~\ref{l2.4} and the definition of strictly binary trees it follows that the proposition is a reformulation of the logic equivalence (i)$\Leftrightarrow$(iii) of Theorem~\ref{th3.3}.
\end{proof}

\begin{defn}\label{d4.1}
Let $(X,d)$ be a metric space with a spectrum $\operatorname{Sp}(X)$ and let $r\in \operatorname{Sp}(X)$ be positive. Denote by $G_{r,X}$ a graph for which $V(G_{r,X})=X$ and
$$
(\{u,v\}\in E(G_{r,X}))\Leftrightarrow (d(u,v)=r), \quad u, v \in X.
$$
For $r=\operatorname{diam} X$ it is clear that $G_{r,X}$ is the diametrical graph of $X$.
\end{defn}

Let $G=(V,E)$ be a nonempty graph, and let $V_0$ be the set (possibly empty) of all isolated vertices of the graph $G$. Denote by $G'$ the subgraph of the graph $G$, generated by the set $V\backslash V_0$.

\begin{prop}\label{p4.3}
Let $(X,d)$ be a finite ultrametric space with $|X|\geq 2$. Then $(X, d) \in \mathfrak{R}$ if and only if for every positive $r\in \Sp{X}$ the graph $G'_{r,X}$ is isometric to the star $K_{1,n-p}$, where $p$ is the level of a node of $T_X$ labeled by $r$ and $n = |X|-1$.
\end{prop}
\begin{proof}
Let $(X, d) \in \mathfrak{R}$ and let $r\in \Sp{X}$ be positive. Let $x_n$ be a leaf of the last level $n$. Consider the path $(x_n,x_{n-1},...,x_0)$ from $x_n$ to the root $x_0 = X$ of $T_X$. Using statement (iii) of Theorem~\ref{th3.3} and properties of representing trees we conclude that $x_{n-1}$, $\ldots$, $x_0$ are the only possible inner nodes of $T_X$ and all labels of $x_n$, $x_{n-1}$, $\ldots$, $x_0$ are different. Hence there is a unique node, say $x_p$, $0\leq p \leq n-1$ labeled by $r$. Let $x'$ be a direct successor of $x_p$ which is leaf and $x''$ be another direct successor of $x_p$. According to Lemma~\ref{l2.3} the equality $d(x,y)=r$ is possible only if $x=x'$ and $y\in L_{T_{x''}}$. Since $|L_{T_{x''}}|=n-p$, the graph $G'_{r,X}$ is isomorphic to $K_{1,n-p}$.

The converse follows from Definition~\ref{d4.1}, statement (iii) of Theorem~\ref{th3.3} and the definition of representing tree $T_X$.
\end{proof}

The following lemma is a reformulation of Theorem~4.1 from~\cite{DDP(P-adic)}.

\begin{lem}\label{l4.5}
Let $(X,d)$ be a finite ultrametric space with $|X|\geq 2$ and let $G_X$ be the diametrical graph of $X$. Then the inequality
\begin{equation}\label{e4.1}
|E(G_X)| \geq |X| -1
\end{equation}
holds. The equality in \eqref{e4.1} occurs if and only if $G_X$ is isomorphic to a star.
\end{lem}

\begin{lem}\label{l4.6}
Let $(X,d)$ be a finite ultrametric space with $|X|\geq 2$. If $(X, d) \in \mathfrak{R}$ then for every $Y\subseteq X$, $|Y|\geq 2$, we have $(Y, d) \in \mathfrak{R}$.
\end{lem}
\begin{proof}
Let $n = |X|$. It is sufficient to prove that $(Y, d) \in \mathfrak{R}$ for the case
$$
Y=X\setminus\{x_i\}, \quad x_i\in X.
$$
Taking into consideration Lemma~\ref{l2.4}, the space $X=\{x_1,...,x_n\}$ for which $T_X$ fulfils statement~(iii) of Theorem~\ref{th3.3} can be the uniquely presented by sequence of nested balls $B_1 \subset B_2 ... \subset B_{n-1} \subset B_{n}$, where $B_1=\{x_1\}$, $B_i=B_{i-1}\cup\{x_i\}$, $i=2,...,n$. Let us consider the set $\mathbf{B}_Y$. It is clear that the following relations hold
\begin{equation}\label{e4.4}
\overline{B}_1 \subset \overline{B}_2 ... \subseteq \overline{B}_{n-2} \subset \overline{B}_{n-1}
\end{equation}
where $\overline{B}_1 = B_1$, ... , $\overline{B}_{i-1} = B_{i-1}$, $\overline{B}_i = B_{i+1}\setminus \{x_i\}$,...,$\overline{B}_{n-1} = B_n\setminus \{x_i\}$.

Relations~(\ref{e4.4}) and Lemma~\ref{l2.4} imply that $T_Y$ fulfils statement~(iii) of Theorem~\ref{th3.3}.
\end{proof}

The next proposition gives us a new characteristic extremal property of spaces $X \in \mathfrak{R}$.

\begin{prop}\label{p4.6}
Let $(X,d)$ be a finite ultrametric space with $|X|\geq 2$. Then the following statements are equivalent.
\begin{itemize}
\item [(i)] $X \in \mathfrak{R}$;
\item [(ii)] The inequality
\begin{equation}\label{e4.2}
|E(G_Y)| \leq |E(G_Z)|
\end{equation}
holds for all $Y \subseteq X$ and all ultrametric spaces $Z$ which satisfy the condition $|Y|=|Z| \geq 2$.
\end{itemize}
\end{prop}
\begin{proof}
Let $X \in \mathfrak{R}$. Then by Lemma~\ref{l4.6} we have $Y \in \mathfrak{R}$ for every $Y \subseteq X$ with $|Y| \geq 2$. Hence $G_Y$ is a star, so that
\begin{equation}\label{e4.3}
|E(G_Y)| = |Y| - 1.
\end{equation}
Lemma~\ref{l4.5} implies
\begin{equation}\label{e4.3*}
|E(G_Z)| \geq |Z| - 1
\end{equation}
if $Z$  is a finite ultrametric space with $|Z| \geq 2$. Now~\eqref{e4.2} follows from \eqref{e4.3}, \eqref{e4.3*} and the equality $|Y|=|Z|$.

Let (ii) hold and let $B \in \mathbf{B}_X$. Condition (ii) and Lemma~\ref{l4.5} imply that the diametrical graph $G_B$ is a star. Hence $X \in \mathfrak{R}$ by Proposition~\ref{p4.1}.
\end{proof}

A graph $G = (V, E)$ together with a function $\omega: E \to \mathbb R^+$ is called a \emph{weighted graph}\/ with the \emph{weight} $\omega$. The weighted graph will be denoted as $(G, \omega)$. In the following we identify a finite ultrametric space $(X, d)$ with a complete weighted graph $(G, \omega_d)$ such that $V(G)=X$ and
$$
\omega_d(\{x,y\}) = d(x,y)
$$
for all distinct $x$, $y \in X$.

\begin{lem}\label{l4.7}
Let $(X, d)$ be an ultrametric space. Then for any cycle $C$ in $(G, \omega_d)$ there exist at least two different edges $e_1$, $e_2 \in E(C)$ such that
$$
\omega_d(e_1) = \omega_d(e_2) = \max_{e \in E(C)} \omega_d (e).
$$
\end{lem}

If $|E(C)|=3$, then Lemma~\ref{l4.7} is a reformulation of the strong triangle inequality. For $|E(C)| \geq 4$, it can be proved by induction on $|E(C)|$. (For details see~\cite[Lemma~1]{DP(MatSb)}.)

For a graph $G = (V, E)$ a \emph{Hamiltonian path} is a path in $G$ that visits every vertex of $G$ exactly once. It is clear that a path $P \subseteq G$ is Hamiltonian if and only if $P$ is a spanning tree of $G$. The following theorem gives us some characterizations of ultrametric spaces $(X,d) \in \mathfrak{R}$ via Hamiltonian paths and spanning stars of $(G, \omega_d)$.

\begin{thm}\label{p4.8}
Let $(X, d)$ be a finite ultrametric space with $|X| \geq 2$. Then the following statements are equivalent.
\begin{itemize}
\item [(i)] $(X, d) \in \mathfrak{R}$.
\item [(ii)] The graph $(G, \omega_d)$ has a Hamiltonian path $P=(x_1, \ldots, x_n)$ such that
\begin{equation}\label{p4.8e1}
\omega_d(\{x_k, x_{k+1}\}) > \omega_d(\{x_{k+1}, x_{k+2}\})
\end{equation}
for $k=1$, $\ldots$, $n-2$, where $n = |X|$.
\item[(iii)] The graph $(G, \omega_d)$ has a spanning star $S$ with
$$
E(S) = \{\{y_0, y_1\}, \ldots, \{y_0, y_{n-1}\}\}
$$
such that
\begin{equation}\label{p4.8e2}
\omega_d(\{y_0, y_i\}) \neq \omega_d(\{y_0, y_j\})
\end{equation}
for distinct $i$, $j \in \{1, \ldots, n-1\}$, where $n = |X|$.
\end{itemize}
\end{thm}
\begin{proof}
\textbf{(i)}$\Rightarrow$\textbf{(ii)}. Let $(X, d) \in \mathfrak{R}$ and let $n = |X|$. By statement~(iii) of Theorem~\ref{th3.3} the representing tree $T_X$ has exactly one inner node at each level expect the last level and, moreover, the last level contains exactly two leaves. Hence the number of the levels of $T_X$ is $n$ and we can enumerate the points of $X$ in the sequence $(x_1, \ldots, x_n)$ such that, for $k=1$, $\ldots$, $n-2$, $\{x_k\}$ is the leaf on the $k$-th level and $\{x_{n-1}\}$, $\{x_{n}\}$ are the leaves of the last level. From the definition of the diametrical graphs it follows that
$$
\omega_d(\{x_1, x_2\}) = d(x_1, x_2) = l(X) = \diam X.
$$
(Recall that $X$ is the root of $T_X$.)

Similarly, for every $k \in \{2, \ldots, n-1\}$, we have
$$
\omega_d(\{x_k, x_{k+1}\}) = d(x_k, x_{k+1}) = l(v_k),
$$
where $v_k$ is the unique inner node at the $k$-th level. Note now that if $v$, $u$ are the nodes of $T_X$ and $v$ is a child of $u$, then the inequality
$$
l(v) < l(u)
$$
holds. It follows directly from statement~(iii) of Theorem~\ref{th3.3} that $v_{k+1}$ is a child of $v_k$ for $k=1$, $\ldots$, $n-2$ and that $v_1$ is a child of $X$. Inequality~\eqref{p4.8e1} follows.

\textbf{(ii)}$\Rightarrow$\textbf{(iii)}. Let $P=(x_1, \ldots, x_n)$ be a Hamiltonian path in $(G, \omega_d)$ such that~\eqref{p4.8e1} holds for $k=1$, $\ldots$, $n-2$. Let us define $y_i = x_{n-i}$ for $i=0$, $\ldots$, $n-1$. Using Lemma~\ref{l4.7} with the cycles $C=C_i$ and $C=C_{i+1}$ such that
$$
E(C_i) = \{\{y_0, y_1\}, \{y_1, y_2\}, \ldots, \{y_{i-1}, y_i\}, \{y_i, y_0\}\}
$$
and
$$
E(C_{i+1}) = \{\{y_0, y_1\}, \{y_1, y_2\}, \ldots, \{y_{i-1}, y_i\}, \{y_{i}, y_{i+1}\}, \{y_{i+1}, y_0\}\}
$$
we obtain
$$
\omega_d(\{y_0, y_i\}) < \omega_d(\{y_0, y_{i+1}\})
$$
for $i=0$, $\ldots$, $n-1$. Statement~(iii) follows.

\textbf{(iii)}$\Rightarrow$\textbf{(i)}. Let (iii) hold. Then, without loss of generality, we can set
$$
X = \{x_1, x_2, \ldots, x_n\}
$$
and
$$
d(x_n, x_1) > d(x_n, x_2)  > \ldots > d(x_n, x_{n-1}).
$$

\begin{figure}[h]
\begin{tikzpicture}[scale=1.2]
\draw (0,0) node [anchor=north] {$x_8$} -- (6,0) node [anchor=north] {$x_1$};
\fill[black] (0,0) circle(2pt); \fill[black] (6,0) circle(2pt);
\draw (0,0) -- (-1,0) node [below] {$x_7$};
\fill[black] (-1,0) circle(2pt);
\draw (0,0) -- (-0.8,1) node [left] {$x_6$};
\fill[black] (-0.8,1) circle(2pt);
\draw (0,0) -- (0,2) node [above] {$x_5$};
\fill[black] (0,2) circle(2pt);
\draw (0,0) -- (1.1,2.5) node [above] {$x_4$};
\fill[black] (1.1,2.5) circle(2pt);
\draw (0,0) -- (3,2) node [above right] {$x_3$};
\fill[black] (3,2) circle(2pt);
\draw (0,0) -- (5,1) node [right] {$x_2$}; \fill[black] (5,1) circle(2pt);
\end{tikzpicture}
\caption{A spanning star in $(G, \omega_d)$ for $(X, d) \in \mathfrak{R}$ with $|X|=8$.}
\label{fig4}
\end{figure}
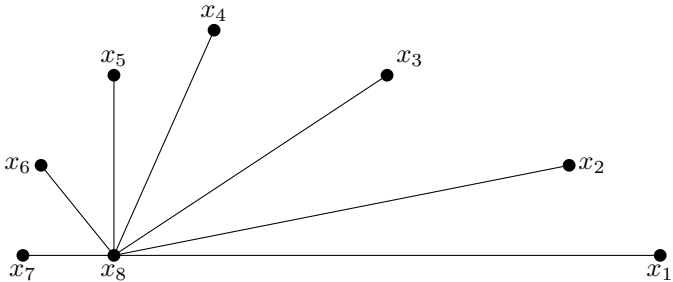

Using the strong triangle inequality we obtain that
$$
d(x_1, x_i) = \diam X \quad\text{and}\quad d(x_i,x_j)<\diam X
$$
for all distinct $i$, $j \in \{2, \ldots, n\}$. Similarly we see that
$$
d(x_k, x_i) = \diam\{x_k, x_{k+1}, \ldots, x_n\}\quad \text{and}\quad d(x_i, x_j) < d(x_k, x_i)
$$
if $i$, $j \in \{k+1, \ldots, n\}$.

Hence statement~(iii) of Theorem~\ref{th3.3} holds. The implication (iii)$\Rightarrow$(i) follows.
\end{proof}

Recall that a cycle $C$ in a graph $G$ is Hamiltonian if $V(C)=V(G)$.

\begin{cor}\label{c4.9}
Let $(X, d)$ be a finite ultrametric space with $|X| \geq 3$. Then $(X,d) \in \mathfrak{R}$ if and only if the weighted graph $(G, \omega_d)$ contains a Hamiltonian cycle $(x_1, \ldots, x_n)$ such that
$$
\omega_d(\{x_1, x_2\}) = \omega_d(\{x_n, x_1\}) = \max_{e \in E(C)} \omega_d(e)
$$
and
$$
\omega_d(\{x_k, x_{k+1}\}) > \omega_d(\{x_{k+1}, x_{k+2}\})
$$
for $k = 1$, $\ldots$, $n-2$.
\end{cor}
The proof is immediate from statement~(ii) of Theorem~\ref{p4.8} and Lemma~\ref{l4.7}. The next lemma is a particular case of Theorem~7 from~\cite{DP(MatSb)}.

\begin{lem}\label{l4.10}
Let $(S, \omega)$ be a weighted star with $\omega(e) > 0$ for every $e \in E(S)$. Then the following conditions are equivalent
\begin{enumerate}
\item There is a unique ultrametric space $(X,d)$ such that $X = V(S)$ and $d(x,y)=\omega(\{x,y\})$ for every $\{x,y\} \in E(S)$.
\item The weight $\omega\colon E(S) \to \mathbb R^+$ is an injective function.
\end{enumerate}
\end{lem}
This lemma and statement~(iii) of Theorem~\ref{p4.8} give us the following.

\begin{cor}\label{c4.11}
Let $(X, d)$ be a finite ultrametric space with $|X| \geq 2$. Then $(X,d) \in \mathfrak{R}$ if and only if $(G, \omega_d)$ contains a spanning star $S$ such that, for every ultrametric $\rho: X\times X \to \mathbb R$, we have
$$
(\forall e \in E(C) : \omega_d(e) = \omega_\rho(e)) \Rightarrow (\rho=d).
$$
\end{cor}

\begin{rem}
Using Theorem~2.5 from~\cite{DPT(P-adic)} and Theorem~7 from~\cite{DP(MatSb)} we can show that the statement

``There is a path $P$ is $(G, \omega_d)$ such that
$$
(\forall e \in E(P) : \omega_d(e) = \omega_\rho(e)) \Rightarrow (\rho=d)
$$
holds for every ultrametric $\rho: X\times X \to \mathbb R$''

is equivalent to
$$
|X|=|\Sp{X}|.
$$
Hence we can not use any Hamiltonian path instead of star in Corollary~\ref{c4.11}.
\end{rem}

Recall that a function $\Phi$ from a metric space $(X, d)$ to a metric space $(Y, \rho)$ is a similarity if there is $\lambda>0$ such that
$$
\lambda(d(x,y)) = \rho(\Phi(x),\Phi(y))
$$
for all $x$, $y \in X$.

\begin{defn}\label{d4.5}
Let  $(X,d)$ and $(Y,\rho)$ be metric spaces. A bijective mapping $\Phi\colon X\to Y$ is a weak similarity if there is a strictly increasing bijection $f\colon \Sp{X}\to \Sp{Y}$ such that the equality
\begin{equation}\label{e4.5}
f(d(x,y))=\rho(\Phi(x),\Phi(y))
\end{equation}
holds for all $x$, $y\in X$.
\end{defn}

If $\Phi \colon X \to Y$ is a weak similarity we say that $X$ and $Y$ are \emph{weakly similar}. If $(X, d)$ is a finite metric space, then every weak similarity $\Phi\colon X \to X$ is an isometry. The notion of weak similarity was introduced in~\cite{DP2} for more general case of semimetric spaces in a slightly different form.

\begin{prop}\label{p4.10}
Let $(X,d) \in \mathfrak{R}$ and let $(Y,\rho)$ be a metric space. If $(X,d)$ and $(Y,\rho)$ are weakly similar, then $(Y,\rho) \in \mathfrak{R}$.
\end{prop}
\begin{proof}
If $\Phi: X \to Y$ is a weak similarity, then $\mathbf{B}_Y = \{\Phi(B): B \in \mathbf{B}_X\}$ and the mapping
$$
\mathbf{B}_X \ni B \mapsto \Phi(B) \in \mathbf{B}_Y
$$
is a bijection. It is clear also that
$$
(A \subseteq C) \Leftrightarrow (\Phi(A) \subseteq\Phi(C))
$$
holds for all $A \subseteq X$ and $C \subseteq X$. Hence the graphs $(\Gamma_X, X)$ and $(\Gamma_Y, Y)$ are isomorphic. Using Theorem~\ref{t2.7} we obtain that $Y$ is ultrametric and $\overline{T}_Y$ is isomorphic to $\overline{T}_X$. Now $Y \in \mathfrak{R}$ follows from statement~(iii) of Theorem~\ref{th3.3}.
\end{proof}

\begin{prop}\label{p4.11}
Let $X$, $Y\in \mathfrak{R}$. Then the following statements are equivalent.
\begin{itemize}
\item [(i)] The trees $\overline{T}_X$ and $\overline{T}_Y$ are isomorphic as rooted trees.
\item [(ii)] $X$ and $Y$ are weakly similar.
\item [(iii)] The equality $|X|=|Y|$ holds.
\end{itemize}
\end{prop}
\begin{proof}
The implication~\textbf{(i)}$\Rightarrow$\textbf{(iii)}. is immediate. Analysis similar to that in the proof of Proposition~\ref{p4.10} shows that~\textbf{(ii)}$\Rightarrow$\textbf{(i)}. holds.

Let us prove \textbf{(iii)}$\Rightarrow$\textbf{(ii)}. Let $|X|=|Y|$. Write $n = |X|=|Y|$. Statement~(ii) of Theorem~\ref{p4.8} implies that there is a Hamiltonian path $(x_1, \ldots, x_n) \subseteq (G, \omega_d)$ and a Hamiltonian path $(y_1, \ldots, y_n) \subseteq (G, \omega_\rho)$ such that
\begin{equation}\label{p4.11e1}
\omega_d(\{x_k, x_{k+1}\}) > \omega_d(\{x_{k+1}, x_{k+2}\})
\end{equation}
and
\begin{equation}\label{p4.11e2}
\omega_\rho(\{x_k, x_{k+1}\}) > \omega_\rho(\{x_{k+1}, x_{k+2}\})
\end{equation}
for $k=1$, $\ldots$, $n-2$. The Gomory-Hu inequality implies that
$$
\Sp{X} = \{d(x_k, x_{k+1}): k=1, \ldots, n-1\} \cup \{0\}
$$
and
$$
\Sp{Y} = \{\rho(y_k, y_{k+1}): k=1, \ldots, n-1\} \cup \{0\}.
$$
Let us define the functions $\Phi: X \to Y$ and $f: \Sp{X} \to \Sp{Y}$ such that
$$
\Phi(x_i)= y_i \quad \text{and}\quad f(0)=0\quad \text{and}\quad f(d(x_k, x_{k+1})) = \rho(x_k, x_{k+1})
$$
for $k=1$, $\ldots$, $n-1$.

Inequality~\eqref{p4.11e1} and \eqref{p4.11e2} imply that $f$ is strictly increasing. Moreover, it is clear that $\Phi$ and $f$ are bijective. Now using Lemma~\ref{l4.7} we obtain that equality~\eqref{e4.5} holds for all $x$, $y \in X$. The implication~(iii)$\Rightarrow$(ii). follows.
\end{proof}

\begin{rem}\label{r4.16}
Let $a$, $b >0$. If $(X, d)$ and $(Y, \rho)$ are ultrametric spaces for which
$$
d(x_1, x_2)=a \quad \text{and} \quad \rho(y_1, y_2)=b
$$
for all distinct $x_1$, $x_2 \in X$ and $y_1$, $y_2 \in Y$, then $(X, d)$ and $(Y, \rho)$ are weakly similar if and only if $|X|=|Y|$. For these spaces we have also $(X, d)\notin \mathfrak{R}$ and  $(Y, \rho) \notin \mathfrak{R}$ if $|X|$, $|Y| \geq 3$.
\end{rem}

It seems to be interesting to find a ``representing tree description'' of the classes of finite ultrametric spaces $X$, $Y$ for which the conditions $|X|=|Y|$, implies that $X$ and $Y$ are weakly similar.


\begin{thebibliography}{99}

\bibitem{BM}
J.~A. Bondy and U.~S.~R. Murty, \emph{{Graph theory}}, Graduate Texts in Mathematics, vol. 244, Springer, New York, 2008.

\bibitem{DLW}
E.~D. Demaine, G.~M. Landau, and O.~Weimann, \emph{{On Cartesian Trees and Range Minimum Queries}}, {Proceedings of the 36th International Colloquium, ICALP 2009, Rhodes, Greece, July 5-12, 2009, Part I}, Lecture notes in Computer Science, vol. 5555, Springer-Berlin-Heidelberg, 2009, pp.~341--353.

\bibitem{Di}
R.~Diestel, \emph{{Graph Theory}}, {Third} ed., Graduate Texts in Mathematics, vol. 173, Springer, Berlin, 2005.

\bibitem{DDP(P-adic)}
D.~Dordovskyi, O.~Dovgoshey, and E.~Petrov, \emph{Diameter and diametrical pairs of points in ultrametric spaces}, p-Adic Numbers Ultrametric Anal. Appl. \textbf{3} (2011), no.~4, 253--262.

\bibitem{DP(MatSb)}
O.~Dovgoshey and E.~Petrov, \emph{Subdominant pseudoultrametric on graphs}, Sb. Math \textbf{204} (2013), no.~8, 1131--1151, translation from Mat. Sb. \textbf{204} (2013), no.~8, 51--72.

\bibitem{DP2}
O.~Dovgoshey and E.~Petrov,\emph{{Weak similarities of metric and semimetric spaces}}, Acta Math. Hungar. \textbf{141} (2013), no.~4, 301--319.

\bibitem{DPT(P-adic)}
O.~Dovgoshey, E.~Petrov, and H.-M. Teichert, \emph{{On spaces extremal for the {G}omory-{H}u inequality}}, p-Adic Numbers Ultrametric Anal. Appl. \textbf{7} (2015), no.~2, 133--142.

\bibitem{Fie}
Miroslav Fiedler, \emph{Ultrametric sets in {E}uclidean point spaces}, Electron. J. Linear Algebra \textbf{3} (1998), 23--30.

\bibitem{GomoryHu(1961)}
R.~E. Gomory and T.~C. Hu, \emph{Multi-terminal network flows}, J. Soc. Indust. Appl. Math. \textbf{9} (1961), 551--570.

\bibitem{GurVyal(2012)}
V.~Gurvich and M.~Vyalyi, \emph{Characterizing (quasi-)ultrametric finite spaces in terms of (directed) graphs}, Discrete Appl. Math. \textbf{160} (2012), no.~12, 1742--1756.

\bibitem{GV}
V.~Gurvich and M.~Vyalyi, \emph{Ultrametrics, trees, flows, and bottleneck arcs}, Mathematical Education, series 3, vol.~16, MCCME, Moscow, 2012, (In Russian).

\bibitem{Hol}
J.~E. Holly, \emph{{Pictures of Ultrametric Spaces, the p-Adic Numbers, and Valued Fields}}, Amer. Math. Monthly \textbf{108} (2001), no.~8, 721--728.

\bibitem{H04}
B.~Hughes, \emph{{Trees and ultrametric spaces: a categorical equivalence}}, Adv. Math. \textbf{189} (2004), no.~1, 148--191.

\bibitem{BH2}
B.~Hughes, \emph{{Trees, ultrametrics, and noncommutative geometry}}, Pure Appl. Math. Q. \textbf{8} (2012), no.~1, 221--312.

\bibitem{WGA}
A.~J. Lemin, \emph{{On {G}elgfand's problem concerning graphs, lattices, and ultrametric spaces}}, AAA62 - Workshop on General Algebra (2001), 12--13.

\bibitem{Lem}
Alex~J. Lemin, \emph{The category of ultrametric spaces is isomorphic to the category of complete, atomic, tree-like, and real graduated lattices {${\bf LAT}^*$}}, Algebra Univers. \textbf{50} (2003), no.~1, 35--49.

\bibitem{PD(UMB)}
E.~Petrov and A.~Dovgoshey, \emph{{On the {G}omory-{H}u inequality}}, J. Math. Sci. \textbf{198} (2014), no.~4, 392--411, Translation from Ukr. Mat. Visn.  \textbf{10} (2013), no. 4, 469--496.

\bibitem{P(TIAMM)}
E.~A. Petrov, \emph{Ball-preserving mappings of finite ulrametric spaces},
  Proceedings of IAMM \textbf{26} (2013), 150--158, (In Russian).

\end{thebibliography}

\end{document}